\documentclass[reqno,draft]{amsart}

\usepackage{pgfplots}
\usepackage{amsthm}
\usepackage{latexsym}
\usepackage{dsfont}
\usepackage{bbm}
\usepackage{amssymb}
\usepackage{amsmath}
\usepackage{mathtools}

\usepackage{todonotes}

\numberwithin{equation}{section}

\theoremstyle{plain}
\newtheorem{Theorem}{Theorem}[section]
\newtheorem{Lemma}[Theorem]{Lemma}
\newtheorem{Cor}[Theorem]{Corollary}

\theoremstyle{remark}
\newtheorem{Rem}[Theorem]{Remark}
\theoremstyle{definition}

\newcommand{\defeq}{\vcentcolon=}
\newcommand{\eqdef}{=\vcentcolon}
\newcommand{\distto}{\stackrel{\mathrm{d}}{\to}}

\newcommand{\N}{\mathbb{N}}

\newcommand{\R}{\mathbb{R}}

\newcommand{\I}{\mathcal{I}}

\newcommand{\imag}{\mathrm{i}}

\newcommand{\Surv}{\mathcal{S}}

\newcommand{\Prob}{\mathbb{P}}
\DeclareMathOperator{\Var}{\mathbb{V}\mathrm{ar}}

\DeclareMathOperator{\sign}{\mathrm{sign}}

\newcommand{\E}{\mathbb{E}}

\newcommand{\F}{\mathcal{F}}
\newcommand{\cZ}{\mathcal{Z}}
\newcommand{\1}{\mathbbm{1}}

\newcommand{\eqdist}{\stackrel{\mathrm{law}}{=}}

\newcommand{\dx}{\mathrm{d} \mathit{x}}

\newcommand{\dProb}{\mathrm{d} \Prob}
\newcommand{\dF}{\mathrm{d} \mathit{F}}
\newcommand{\dL}{\mathrm{d} \mathit{L}}

\newcommand{\tofd}{\overset{{\rm f.d.d.}}{\longrightarrow}}

\newcommand{\eee}{{\rm e}}

\begin{document}

\title[Stable-like fluctuations of Biggins' martingales]{Stable-like fluctuations of Biggins' martingales}
\author{Alexander Iksanov}
\address{Faculty of Computer Science and Cybernetics, Taras Shevchenko National University of Kyiv, 01601 Kyiv, Ukraine}
\email{iksan@univ.kiev.ua}
\author{Konrad Kolesko} \address{Universit\"at Innsbruck, Austria
    and
    Uniwersytet Wroc\l{}awski, Wroc\l{}aw, Poland}\email{kolesko@math.uni.wroc.pl}
\author{Matthias Meiners}
\address{Universit\"at Innsbruck, Austria} \email{matthias.meiners@uibk.ac.at}

\subjclass[2010]{Primary: 60J80. Secondary: 60F05, 60G42}
\keywords{Autoregressive process; Biggins' martingale, branching
random walk; martingale; stable distribution}

\begin{abstract}
Let $(W_n(\theta))_{n \in \N_0}$ be Biggins' martingale
associated with a supercritical branching random walk,
and let $W(\theta)$ be its almost sure limit.
Under a natural condition for the offspring point process in the branching random walk,
we show that if the law of $W_1(\theta)$ belongs to the domain of normal attraction
of an $\alpha$-stable distribution for some $\alpha \in (1,2)$,
then, as $n\to\infty$, there is weak convergence of the tail process
$(W(\theta) - W_{n-k}(\theta))_{k \in \N_0}$,
properly normalized, to a random scale multiple
of a stationary autoregressive process of order one with $\alpha$-stable marginals.
\end{abstract}

\maketitle

\section{Introduction and main result}  \label{sec:Intro and main results}

\subsection{Introduction}

The branching random walk on the real line is a model for the evolution of a population with a spatial component.
It has connections to classical objects of statistical physics such as directed polymers on disordered trees
\cite{Derrida+Spohn:1988} to give just one example;
we refer to the recent lecture notes \cite{Shi:2015} for further examples and references.

Certain nonnegative martingales, the \emph{additive martingales},
are key tools in the description and analysis of
the asymptotic behavior of the branching random walk
such as the spread of particles at typical positions,
see e.g.\;\cite{Biggins:1992}.
These martingales are sometimes called \emph{Biggins' martingales}
in honor of Biggins' seminal contribution \cite{Biggins:1977},
in which conditions for the convergence of these martingales
to nondegenerate limits were found.
It is then natural to ask for the speed of convergence.

In the present paper, we are interested in the rate of convergence of Biggins' martingale
in the case where the martingale at time $1$ has a power tail.
Requiring only minimal assumptions,
we prove convergence of the finite-dimensional distributions
of the tail of Biggins' martingale, suitably normalized,
to a randomly scaled stationary autoregressive process of order one
with stable marginals.

\subsection{Model description}

A (one-dimensional) branching random walk is a particle system on the real line.
At time $n=0$ it consists of one particle, the ancestor, located at the origin.
At time $n=1$ the ancestor produces offspring (the first generation)
the positions of which are given by the points of a point process
$\cZ = \sum_{j=1}^N \delta_{X_j}$ on $\R$.
The number of offspring, $N = \cZ(\R)$, is a random variable
taking values in $\N_0 \cup \{+\infty\} = \{0,1,2,\ldots\} \cup \{+\infty\}$.
At time $n=2$, the individuals of the first generation
produce offspring, the second generation,
with displacements with respect to their mothers' positions
given by independent copies of the point process $\cZ$.
The further generations are formed analogously.

More formally, let $\I=\bigcup_{n\in\N_0}\N^n$ be the set of all
possible individuals. The ancestor label is the empty word
$\varnothing$, its position is $S(\varnothing)=0$. On some
probability space $(\Omega, \F, \Prob)$ let $(\cZ(u))_{u \in \I}$
be a family of independent and identically distributed (i.i.d.)\
copies of the point process $\cZ$. An individual of the $n$th
generation with label $u = u_1\ldots u_n$ and position $S(u)$
produces a random number $N(u)$ of offspring at time $n+1$. The
offspring of the individual $u$ are placed at random locations on
$\R$ given by the positions of the point process
\begin{equation*}
\delta_{S(u)} * \cZ(u) = \sum_{j=1}^{N(u)} \delta_{S(u) + X_j(u)}
\end{equation*}
where $\cZ(u) = \sum_{j=1}^{N(u)} \delta_{X_j(u)}$ and $N(u)$ is
the number of points in $\cZ(u)$. The offspring of the individual
$u$ are enumerated by $uj = u_1 \ldots u_n j$, where
$j=1,\ldots,N(u)$ (if $N(u)<\infty$) or $j=1,2,\ldots$ (if
$N(u)=\infty$), and the positions of the offspring are denoted by
$S(uj)$. No assumptions are imposed on the dependence structure of
the random variables $N(u), X_1(u),X_2(u),\ldots$ for fixed
$u\in\I$. The point process of the positions of the $n$th
generation individuals will be denoted by $\cZ_n$ so that
$\cZ_0=\delta_0$ and
\begin{equation*}
\cZ_{n+1} = \sum_{|u|=n} \sum_{j=1}^{N(u)} \delta_{S(u)+X_j(u)},
\end{equation*}
where here and hereafter, $|u|=n$ means that the sum is taken over
all individuals of the $n$th generation rather than over all
$u\in\N^n$. The sequence of point processes $(\cZ_n)_{n \in \N_0}$
is then called a \emph{branching random walk} (BRW).

We assume throughout that $(\cZ_n)_{n \in \N_0}$ is \emph{supercritical}, i.e., ${\E[N]>1}$.
This implies $\Prob(\Surv)>0$ where $\Surv = \{\cZ_n(\R)>0 \text{ for every } n \in \N_0\}$.
The sequence of generation sizes in the BRW, $(\cZ_n(\R))_{n \in \N_0}$, forms a Galton--Watson process
if $\Prob(N<\infty)=1$.

Consider the Laplace transform of the intensity measure $\mu(\cdot) \defeq \E[\cZ(\cdot)]$ of $\cZ$,
\begin{equation*}    \label{eq:m}
m: \R \to [0,\infty],   \qquad
\theta \mapsto \int_{\R} \eee^{-\theta x} \, \mu(\dx) = \E \bigg[ \int_{\R} \eee^{-\theta x} \,\cZ(\dx)\bigg].
\end{equation*}
We assume that $m(\theta)< \infty$ for some $\theta\in\R$. For
each such $\theta$, let
\begin{equation*}\label{eq:W_n}
W_n(\theta) \defeq     \frac1 {m(\theta)^n} \int_{\R} \eee^{-\theta x}
\,\cZ_n(\dx) = \frac 1 {m(\theta)^n}\sum_{|u|=n} e^{-\theta S(u)}, \quad n\in\N_0.
\end{equation*}
We write $|u|<n$ if $u \in \N^k$ for some $k<n$
and set $\F_n=\sigma(\cZ(u)\colon |u|<n)$,
the $\sigma$-algebra generated by the first $n$ generations.
It is well-known that, for every $\theta$ with $m(\theta)<\infty$,
$(W_n(\theta))_{n \in\N_0}$ forms a nonnegative
martingale with respect to $(\F_n)_{n\in \N_0}$
and thus converges almost surely to a random variable $W(\theta)$
satisfying $\E[W(\theta)] \leq 1$.
This martingale is called \emph{additive} or \emph{Biggins' martingale}.

\subsection{The main result}    \label{subsec:main result}

Next, we introduce an object that appears in our main result. Let
$(U_k)_{k \in \N_0}$ denote a stationary autoregressive process of
order $1$ with parameter $\varphi \in (0,1)$ defined by
\begin{equation}    \label{eq:U_k recursion}
U_k = \varphi U_{k-1} + Q_k,    \quad   k \in \N
\end{equation}
where $U_0$ is independent of the sequence $Q_1,Q_2,\ldots$
of i.i.d.\ random variables which have characteristic function
\begin{equation}    \label{eq:chf of Q_1}
\E\big[e^{\imag t Q_k}\big]
=\exp\!\bigg(\frac{\Gamma(2\!-\!\alpha)}{\alpha-1}c|t|^\alpha\Big(\!\cos\!\Big(\frac{\pi\alpha}{2}\Big)-\imag \sin\!\Big(\frac{\pi\alpha}{2}\Big)\sign(t) \Big) \bigg),\;t \in \R
\end{equation}
for some $c>0$, where $\Gamma(\cdot)$ is the gamma function.
Notice that the $Q_k$ have spectrally positive $\alpha$-stable laws.
Observe that, for $t \in \R$,
\begin{equation}    \label{eq:chf U_0}
\E\big[e^{\imag t U_0}\big]
=
\prod_{j\geq 0}\E\big[e^{\imag \varphi^j tQ_1}\big]
=
\exp\!\bigg(\frac{\Gamma(2\!-\!\alpha)}{\alpha-1}\frac{c|t|^\alpha}{1-\varphi^\alpha}
\Big(\!\cos\!\Big(\frac{\pi\alpha}{2}\Big)- \imag \sin\!\Big(\frac{\pi\alpha}{2}\Big)\sign(t) \Big)\! \bigg).
\end{equation}

Our main result is the following theorem.

\begin{Theorem} \label{Thm:stable fluctuations of Biggins's martingale}
Suppose there exist $\alpha \in (1,2)$ and $c>0$ such that
\begin{equation}    \label{eq:contraction assumption}
\kappa\defeq \tfrac{m(\alpha\theta)}{m(\theta)^\alpha}<1
\end{equation}
and
\begin{equation}    \label{eq:tail assumption W_1}
\Prob(W_1(\theta)>x) \sim cx^{-\alpha}  \quad   \text{as } x \to \infty.
\end{equation}
Further, let $(U_r)_{r \in \N_0}$ be independent of $W(\theta)$
and defined as in \eqref{eq:U_k recursion} with
$\varphi\!=\!\kappa^{1/\alpha}$. Let $c$ in \eqref{eq:chf of Q_1}
be the same as in \eqref{eq:tail assumption W_1}. Then, with
$W_j(\theta)=1$ for $j<0$, we have
\begin{equation}    \label{eq:stable fluctuations of Biggins's martingale}
\big(\kappa^{-(n-r)/\alpha}(W(\theta) -
W_{n-r}(\theta))\big)_{r\in \N_0} \tofd
W(\alpha\theta)^{1/\alpha}(U_r)_{r\in\N_0}      \quad   \text{as }
n \to \infty
\end{equation}
where $\tofd$ denotes convergence of the finite-dimensional distributions.

\end{Theorem}

\begin{Rem} \label{Rem:trivial limit}
Without further assumptions,
the martingale convergence theorem implies that $W(\theta) \defeq \lim_{n \to \infty} W_n(\theta)$ exists almost surely,
but $\Prob(W(\theta)\!=\!0)=1$ may hold.
However, the assumptions of Theorem \ref{Thm:stable fluctuations of Biggins's martingale}
guarantee $\E[W(\theta)]=1$.
More precisely, notice that $p \mapsto m_\theta(p) \defeq m(p \theta)/m(\theta)^p$ is convex
with $m_\theta(1)=1$ and $m_\theta(\alpha)=\kappa < 1$.
Thus $m_\theta'(1) < 0$, which gives $\theta m'(\theta)/m(\theta) - \log (m(\theta)) \in [-\infty,0)$.
Further, $\E[W_1(\theta) \log^+(W_1(\theta))] < \infty$ is a consequence of \eqref{eq:tail assumption W_1}.
Therefore, the main result of \cite{Lyons:1997} together with the subsequent remark give $\E[W(\theta)]=1$.

On the other hand, the assumptions of our main result do not rule out the case where $\Prob(W(\alpha \theta)\!=\!0)=1$.
In this situation, Theorem~\ref{Thm:stable fluctuations of Biggins's martingale} remains valid,
but the limit process in~\eqref{eq:stable fluctuations of Biggins's martingale} is trivial.
\end{Rem}

Specializing Theorem~\ref{Thm:stable fluctuations of Biggins's
martingale} for $r=0$, we obtain the following one-dimensional
result.

\begin{Cor} \label{Cor:stable fluctuations of Biggins's martingale}
Under the assumptions of Theorem \ref{Thm:stable fluctuations of Biggins's martingale},
\begin{equation*}
\kappa^{-n/\alpha}(W(\theta) - W_n(\theta))
\distto W(\alpha\theta)^{1/\alpha} U_0  \quad   \text{as } n \to \infty
\end{equation*}
where, for $t \in \R$,
\begin{align*}
\E\big[e^{\imag t W(\alpha\theta)^{\frac1\alpha}U_0}\big] =
\E\bigg[\exp\!\bigg(\frac{\Gamma(2\!-\!\alpha)}{\alpha-1}\frac{cW(\alpha\theta)}{1-\kappa}|t|^\alpha\Big(\!\cos
\Big(\frac{\pi\alpha}{2}\Big)-\imag
\sin\!\Big(\frac{\pi\alpha}{2}\Big)\sign(t) \Big) \bigg)\bigg].
\end{align*}
\end{Cor}

The limit distribution in Corollary \ref{Cor:stable fluctuations of Biggins's martingale} is a scale mixture of $\alpha$-stable laws.

\subsection{Related literature} \label{subsec:literature}

Rate of convergence results in the form of a central limit theorem
and a law of the iterated logarithm are given in \cite{Iksanov+Kabluchko:2016},
see also \cite{Hartung+Klimovsky:2017} for a recent interesting contribution in the setting of branching Brownian motion.
There are various earlier results, but here we confine ourselves to referring to \cite[p.\;1182]{Iksanov+Kabluchko:2016}
for a thorough account of the literature.

The counterpart of our Corollary \ref{Cor:stable fluctuations of Biggins's martingale}
for the Galton--Watson process was proved in \cite{Heyde:1971}.
In the setting of weighted branching processes, which includes the branching random walk as a special case,
an analogue of our Corollary \ref{Cor:stable fluctuations of Biggins's martingale} was obtained in
\cite{Roesler+al:2002} (since \cite{Roesler+al:2002} is not easily available
we also refer to the conference paper \cite{Roesler+al:2002b},
which is an abridged version of \cite{Roesler+al:2002})
under the assumption $m((\alpha+\varepsilon)\theta)<\infty$ for some $\varepsilon>0$.
This assumption is not required here.

\subsection{Heuristics} \label{subsec:Heuristics}

We continue with an informal discussion of why Theorem \ref{Thm:stable fluctuations of Biggins's martingale}
should be true. From the representation of $W_{n+j}(\theta) - W_{n+j-1}(\theta)$ as a random weighted sum of i.i.d.\ copies of $W_1(\theta)-1$
and the limit theory for independent, infinitesimal triangular arrays it is plausible that
\begin{align*}
\big(\kappa^{-n/\alpha}(W_{n+j}(\theta) -
W_{n+j-1}(\theta))\big)_{j \in \N}
&=
\Big(\kappa^{(j-1)/\alpha}\frac{W_{n+j}(\theta)-W_{n+j-1}(\theta)}{\kappa^{(n+j-1)/\alpha}}\Big)_{j\in\N}   \\
&\tofd
W(\alpha\theta)^{1/\alpha}(\kappa^{(j-1)/\alpha}Q_j)_{j\in\N}.
\end{align*}
In view of this one may expect that, for fixed $r \in \N$ as $n
\to \infty$,
\begin{align*}
\kappa^{-(n-r)/\alpha}(W(\theta) - W_{n-r}(\theta))
&= \kappa^{-(n-r)/\alpha} \sum_{j \geq 1}(W_{n-r+j}(\theta)-W_{n-r+j-1}(\theta))    \\
&\distto
W(\alpha\theta)^{1/\alpha}\sum_{j\geq 1}\kappa^{(j-1)/\alpha}Q_j
\eqdist W(\alpha\theta)^{1/\alpha} \, U_0.
\end{align*}
Similarly, for $r_1,r_2 \in \N_0$, $r_1<r_2$ one would expect that
\begin{align*}
&\bigg(\frac{W(\theta)-W_{n-r_1}(\theta)}{\kappa^{(n-r_1)/\alpha}}, \frac{W(\theta)-W_{n-r_2}(\theta)}{\kappa^{(n-r_2)/\alpha}}\bigg)   \\
&=
\bigg(\frac{W(\theta)-W_{n-r_1}(\theta)}{\kappa^{(n-r_1)/\alpha}},
\kappa^{(r_2-r_1)/\alpha}\frac{W(\theta)-W_{n-r_1}(\theta)}{\kappa^{(n-r_1)/\alpha}}    \\
&\hphantom{= \bigg(}
+\kappa^{(r_2-r_1-1)/\alpha}\frac{W_{n-r_1} (\theta)-W_{n-r_1-1}(\theta)}{\kappa^{(n-r_1-1)/\alpha}}+\ldots+\frac{W_{n-r_2+1}(\theta)-W_{n-r_2}
(\theta)}{\kappa^{(n-r_2)/\alpha}}\bigg) \\
&\distto W(\alpha\theta)^{1/\alpha} (U_0, \kappa^{(r_2-r_1)/\alpha}U_0+\kappa^{(r_2-r_1-1)/\alpha}Q_1+\ldots+Q_{r_2-r_1})   \\
&= W(\alpha\theta)^{1/\alpha} (U_0,U_{r_2-r_1}) \eqdist
W(\alpha\theta)^{1/\alpha} (U_{r_1}, U_{r_2})
\end{align*}
having utilized the stationarity of $(U_r)_{r\in\N_0}$ for the
last distributional equality.

\subsection{Examples}   \label{subsec:examples}

Typically, our main result applies when the number of offspring $N$ has
a heavy tail while the displacements $X_j$ are `tame'. For instance, if
\begin{equation}\label{rel}
\Prob(N>x)\sim d x^{-\alpha}\quad   \text{as } x \to \infty
\end{equation}
for some $\alpha \in (1,2)$ and $d>0$, and if $X_1,X_2,\ldots$ is a sequence of i.i.d.\
random variables independent of $N$ such that condition \eqref{eq:contraction assumption} holds, that is,
\begin{equation}\label{cond}
\E[e^{-\alpha\theta X_1}]
< (\E[N])^{\alpha-1} (\E[e^{-\theta X_1}])^\alpha < \infty,
\end{equation}
then \eqref{eq:tail assumption W_1} holds according to Proposition
4.3 in \cite{Fay+al:2006}. In particular, condition \eqref{cond} is satisfied for all
sufficiently small $\theta > 0$ if the $X_j$ have a standard normal law.

On the other hand, one may wonder whether there are point processes $\cZ$ with 
infinitely many points satisfying the assumptions \eqref{eq:contraction
assumption} and \eqref{eq:tail assumption W_1} of Theorem
\ref{Thm:stable fluctuations of Biggins's martingale}. In \cite{Iksanov+Polotskiy:2006} it is
demonstrated that \eqref{eq:contraction assumption} and
\eqref{eq:tail assumption W_1} are incompatible if
$N=\cZ(\R)=\infty$ almost surely and $\cZ$ is either an
inhomogeneous Poisson process or a point process with independent
points. Now we show that a slight modification of the example given in the first paragraph of the section leads to a point
process $\cZ$ with $\Prob(N=\infty)=1$ which satisfies the
assumptions of Theorem \ref{Thm:stable fluctuations of Biggins's
martingale}. Let $K$ be a random variable taking positive integer
values with the same tail behavior as in \eqref{rel}. Further, let $Y_1, Y_2, \ldots$ be independent copies of a positive random variable $Y$
such that the sequence $(Y_k)_{k \in \N}$ is independent of $K$.
For some $a>0$ to be specified below, set
\begin{equation*}
X_k \defeq  Y_k \1_{\{K\geq k\}} + ak \1_{\{K<k\}}, \quad   k \in \N.
\end{equation*}
Increasing $d$ if necessary we can assume that $\E[K]>1$
and then pick $\theta>0$ and $a$ such that
\begin{equation*}
m(\theta) = \E \bigg[\sum_{k \geq 1}e^{-\theta X_k}\bigg]
= \E[K] \E\big[e^{-\theta Y}\big]+(1-e^{-\theta a})^{-1} \E\big[e^{-\theta a(K+1)}\big]
=1.
\end{equation*}
This entails
\begin{equation*}
\kappa
=m(\alpha\theta)
=\E\bigg[\sum_{k \geq 1}e^{-\alpha \theta X_k}\bigg]
= \E[K] \E\big[e^{-\alpha \theta Y}\big]+(1\!-\!e^{-\alpha \theta a})^{-1} \E\big[e^{-\alpha \theta a(K+1)}\big]
<1,
\end{equation*}
so that \eqref{eq:contraction assumption} holds.
By Proposition 4.3 in \cite{Fay+al:2006}
\begin{equation*}
\Prob\bigg(\sum_{k\geq 1}e^{-\theta Y_k}\1_{\{K\geq k\}}>x\bigg)
\sim
(\E[e^{-\theta Y}])^\alpha \, d  x^{-\alpha}  \quad   \text{as } x\to\infty.
\end{equation*}
Since $\sum_{k \geq 1} e^{-\theta ak}\1_{\{K<k\}}=(1-e^{-\theta a})^{-1}e^{-\theta a(K+1)}$
is almost surely nonnegative and bounded, we infer
\begin{align*}
\Prob(W_1(\theta)>x)
&=
\Prob\bigg(\sum_{k\geq 1}e^{-\theta Y_k}\1_{\{K\geq k\}}+\sum_{k\geq 1}e^{-\theta ak}\1_{\{K<k\}}>x \bigg)
\sim
(\E[e^{-\theta Y}])^\alpha d x^{-\alpha}
\end{align*}
as $x \to \infty$, that is, \eqref{eq:tail assumption W_1} holds.

\section{Tail behavior in the branching random walk}

An important ingredient in the proof of Theorem \ref{Thm:stable fluctuations of Biggins's martingale} is the following result
on the tail behavior of the martingale $(W_n(\theta))_{n \in \N_0}$,
which we believe is interesting in its own right. As usual, for a real number $x$, we define $x^{\pm} \defeq (\pm x) \vee 0$.

\begin{Theorem} \label{Thm:tail of W(theta)}
Suppose there exist $\alpha \in (1,2)$,
$\varepsilon>0$ and a function $\ell$ slowly varying at $\infty$
such that \eqref{eq:contraction assumption} holds, that
\begin{equation}    \label{eq:m((alpha+epsilon)theta)<infty}
m((\alpha+\varepsilon)\theta)<\infty
\end{equation}
and that
\begin{equation}    \label{eq:W_1(theta) reg varying tail}
\Prob(W_1(\theta)>x)~\sim~ x^{-\alpha}\ell(x)   \quad   \text{as }  x \to \infty.
\end{equation}
Then, for any bounded sequence $(a_j)_{j \in \N_0}$, the series $\sum_{j\geq 0} a_j (W_{j+1}(\theta)-W_j(\theta))$ converges almost surely and in $L_p$ for $p\in [1,\alpha)$. Furthermore, as $x\to\infty$,
\begin{align}
\Prob\bigg(\sum_{j\geq 0} a_j (W_{j+1}(\theta)-W_j(\theta))>x\bigg)
&\sim   \textstyle
\sum_{j\geq 0} \kappa^j(a_j^+)^{\alpha}\, \Prob(W_1(\theta)>x)        \label{eq:tail <grad W,a>}
\end{align}
and
\begin{align}
\Prob\bigg(\sum_{j\geq 0} a_j(W_{j+1}(\theta)-W_j(\theta))<-x\bigg)
&\sim   \textstyle
\sum_{j\geq 0} \kappa^j (a_j^-)^{\alpha}\, \Prob(W_1(\theta)>x).  \label{eq:tail <grad W,a->}
\end{align}
If \eqref{eq:W_1(theta) reg varying tail} holds with $\lim_{x\to\infty}\,\ell(x)=c$ for some $c>0$, that is,
if \eqref{eq:tail assumption W_1} holds, then \eqref{eq:contraction assumption} is
sufficient for \eqref{eq:tail <grad W,a>} and \eqref{eq:tail <grad W,a->} (i.e., \eqref{eq:m((alpha+epsilon)theta)<infty} is
not needed).
\end{Theorem}
\begin{Rem}
Since $W_0(\theta)=1$ almost surely, \eqref{eq:tail <grad W,a>} with $a_j=1$ for $j\in\N_0$ yields
\begin{equation}    \label{eq:tail W}
\Prob(W(\theta)>x)
\sim (1-\kappa)^{-1} \Prob(W_1(\theta)>x)   \quad   \text{as } x \to \infty.
\end{equation}
This relation can be found in earlier literature in various guises.
If $\Prob(N < \infty)=1$, then $(W_n(0))_{n \in \N_0}$ is a
supercritical normalized Galton-Watson process.
In this case, \eqref{eq:tail W} was proved in \cite{Bingham+Doney:1974} for non-integer
$\alpha>1$ and in \cite{deMeyer:1982} for integer $\alpha\geq 2$.
If $\theta>0$, $\Prob(N<\infty)=1$ and
$\cZ((-\infty,-\theta^{-1}\log m(\theta)))=0$ almost surely,
$W(\theta)$ can be viewed as a limit random variable in the
Crump-Mode branching process. In this case, \eqref{eq:tail W} was
obtained in \cite{Bingham+Doney:1975} for non-integer $\alpha>1$.
In the setting of the branching random walks a proof of relation
\eqref{eq:tail W} was sketched in \cite{Liang+Liu:2011}.
A complete proof for non-integer $\alpha>1$ along similar lines was given in 2003 in
an unpublished diploma paper of Polotskiy (Kyiv).
The techniques exploited in the aforementioned works are based on Laplace-Stieltjes transforms
and Abelian and Tauberian theorems.
In the more general setting of weighted branching processes limit
theorems for triangular arrays were exploited in
\cite{Roesler+al:2002} to prove \eqref{eq:tail W} under the
extra assumption that the positions of the first generation
individuals are almost surely bounded. An alternative probabilistic proof
of \eqref{eq:tail W} based on martingale theory was given in
\cite{Iksanov+Polotskiy:2006}. Unfortunately, this proof is flawed,
and one purpose of the present paper is to give a correct
probabilistic proof of \eqref{eq:tail W} under optimal assumptions.
\end{Rem}

The rest of the paper is organized as follows. Theorems
\ref{Thm:tail of W(theta)} and \ref{Thm:stable fluctuations of
Biggins's martingale} are proved in Sections \ref{sec:tail of
W(theta)} and \ref{sec:stable fluctuations of Biggins's
martingale}, respectively. 

\section{Proof of Theorem~\ref{Thm:tail of W(theta)}}\label{sec:tail of W(theta)}
Henceforth, we shall abbreviate $W_n(\theta)$ and $W(\theta)$ by $W_n$ and $W$, respectively.
Set $Y_u \defeq e^{-\theta S(u)}/m^{|u|}(\theta)$ for $u\in \I$, so that
$W_n=\sum_{|u|=n}Y_u$ for $n\in \N_0$.

Under the assumptions of Theorem \ref{Thm:tail of W(theta)},
the function $m_\theta(p) \defeq \E\big[\sum_{|u|=1} Y_u^p\big]$ is log-convex on $(1,\alpha)$,
$m_\theta(1)=1$ and $m_\theta(\alpha) = \kappa < 1$. Hence,
$m_\theta(p) < 1$ for all $p \in (1,\alpha)$.
We can thus choose $\delta \in (0,\alpha-1)$ such that
$m_\theta(\alpha+\delta)<1$ and further
\begin{equation}    \label{eq:m_theta(alpha+-delta)}
\E \bigg[\sum_{|u|=n}Y_u^{\alpha-\delta}\bigg] = m_\theta(\alpha-\delta)^n < 1
\quad   \text{and}  \quad
\E \bigg[\sum_{|u|=n}Y_u^{\alpha+\delta}\bigg] = m_\theta(\alpha+\delta)^n <1.
\end{equation}
The second inequality in \eqref{eq:m_theta(alpha+-delta)} implies in particular that
\begin{equation}    \label{eq:sum_u Y_u^p < infty}
\sum_{|u|=n} Y_u^p < \infty \quad   \text{a.\,s.}
\end{equation}
for all $p \in [1,\alpha+\delta]$.

For $k \in \N_0$, the random variable $W_k$ is a function of the
family $(\cZ_v)_{v \in \I}$. For any $u \in \I$, we define
$W_k^{(u)}$ to be the same function applied to the family
$(\cZ_{uv})_{v \in \I}$, and $W^{(u)} \defeq \lim_{k \to \infty}
W_k^{(u)}$ a.s. We shall use the decomposition
\begin{equation*}
W_{n+1}-W_n = \sum_{|u|=n} Y_u (W_1^{(u)}-1).
\end{equation*}
Observe that the $Y_u$, $|u|=n$ are
$\F_n$-measurable, whereas the $W_1^{(u)}$, $|u|=n$ are i.i.d., independent of $\F_n$
and have the same law as $W_1$.
In what follows, we write $\Prob_n(\cdot)$ and $\E_n[\cdot]$
for $\Prob(\cdot | \F_n)$ and $\E[ \cdot | \F_n]$, respectively,
and set $F(x) \defeq \Prob(|W_1-1 | \leq x)$, $x\in\R$.

Put $R_n\defeq \sum_{j=0}^n a_j (W_{j+1}-W_j)$ for $n\in \N_0$. The sequence $(R_n, \F_{n+1})_{n\in\N_0}$ is a martingale. To ensure that the martingale converges a.s.\ and in $L_p$ for $p\in (1,\alpha)$ it suffices to show that it is $L_p$-bounded. The $L_p$-boundedness follows from 
\begin{eqnarray*}
\sup_{n\geq 0}\E [|R_n|^p]
&\leq&
4\sum_{n\geq 1}\E[|R_n-R_{n-1}|^p]+\E[|R_0|^p] \leq 4 \sum_{n \geq 0}|a_n|^p \E[|W_{n+1}-W_n|^p]	\\
&=&
4\sum_{n\geq 0} |a_n|^p \E \Big[\E_n \Big|\sum_{|u|=n}Y_u(W_1^{(u)}-1)\Big|^p\Big]	\\
&\leq&
16\E[|W_1-1|^p] \sum_{n \geq 0} |a_n|^p m_\theta(p)^n<\infty 
\end{eqnarray*}
where the first and third inequalities are obtained with the help of the Topchii-Vatutin inequality for martingales
\cite[Theorem 2]{Topchii+Vatutin:1997}, and $\E[|W_1-1|^p]<\infty$ is a consequence of \eqref{eq:W_1(theta) reg varying tail}.

Throughout the rest of this section we assume, without loss of generality, that $\sup_{j\geq 0} |a_j| \leq 1$. Passing to the proof of \eqref{eq:tail <grad W,a>} we first show that there exists some $x_0>0$
that does not depend on $n$ such that for all $x \geq x_0$, we have
\begin{equation}    \label{eq:basic tail estimate <grad W,a> under P_n uniform}
\frac{\Prob_n(|\sum_{j \geq n} a_j (W_{j+1}-W_j)|>x)}{1-F(x)}
\leq C \sum_{j \geq n} |a_j|^{\alpha-\delta} \E_n[\Xi_j]\quad \text{a.\,s.}
\end{equation}
where $C$ is a finite, deterministic constant that does not depend on $n$ or $x_0$ and
\begin{equation} \label{eq:Xi_n}
\Xi_n =
\sum_{|u|=n}Y_u^{\alpha-\delta}+\sum_{|u|=n}Y_u^{\alpha+\delta}
\end{equation}
for some $\delta$ satisfying \eqref{eq:m_theta(alpha+-delta)}.
Note that
\begin{equation}    \label{eq:E[Xi_n]}
\E[\Xi_n] = m_\theta(\alpha-\delta)^n + m_\theta(\alpha+\delta)^n < \infty
\quad   \text{and}  \quad
\E\bigg[\sum_{n\geq 0} \Xi_n\bigg] < \infty.
\end{equation}
For typographical ease, set $Q\defeq W_1-1$, $Q_u\defeq W_1^{(u)}-1$ and $Y_{u,a}\defeq a_{|u|}Y_u$.
For any fixed $n\in \N_0$ and $x>0$, we infer
\begin{align*}
&\Prob_n\bigg(\Big|\sum_{j \geq n} a_j(W_{j+1}-W_j)\Big|>x\bigg)    \\
&~=
\Prob_n\bigg(\Big|\sum_{|u| \geq n}Y_{u,a}Q_u\Big|>x,\, \sup_{|u| \geq n}|Y_{u,a}Q_u|>x \bigg)  \\
&\hphantom{~=~}+
\Prob_n\bigg(\Big|\sum_{|u| \geq n}Y_{u,a}Q_u\Big|>x,\, \sup_{|u| \geq n} |Y_{u,a}Q_u| \leq x \bigg)    \\
&~\leq
\Prob_n\bigg(\sup_{|u| \geq n}|Y_{u,a}Q_u| >x\bigg)+\Prob_n\bigg(\Big|\sum_{|u| \geq n}Y_{u,a} Q_u \1_{\{|Y_{u,a}Q_u|\leq x\}}\Big|>x\bigg).
\end{align*}
We set $\Prob_n\big(\sup_{|u| \geq n}|Y_{u,a}Q_u| >x\big) \eqdef I_1(n,x)$ and
\begin{align*}
\Prob_n&\bigg(\Big|\sum_{|u| \geq n}Y_{u,a} Q_u \1_{\{|Y_{u,a}Q_u|\leq x\}}\Big|>x\bigg)    \\
&=~
\Prob_n\bigg(\Big|\sum_{|u| \geq n}\big(Y_{u,a} Q_u \1_{\{|Y_{u,a}Q_u|\leq x\}}-\E_{|u|}[Y_{u,a}Q_u\1_{\{|Y_{u,a}Q_u|\leq x\}}]\big)\Big|>\frac x2 \bigg)   \\
&\hphantom{=~}
+ \Prob_n\bigg(\Big|\sum_{|u| \geq n}\E_{|u|}[Y_{u,a}Q_u\1_{\{|Y_{u,a}Q_u|\leq x\}}]\Big|>\frac x2\bigg)
\eqdef  I_2(n,x)+I_3(n,x).
\end{align*}
Put $T(x)\defeq \int_{[0,\,x]} y^2 \, \dF(y)$ and $R(x)\defeq \int_{(x,\infty)} y \, \dF(y)$ for $x>0$.
By Karamata's theorem (Theorems 1.6.4 and 1.6.5 in \cite{Bingham+Goldie+Teugels:1989})
\begin{equation*}   \textstyle
T(x) \sim \frac{\alpha}{2-\alpha} x^2(1-F(x)) \sim \frac{\alpha}{2-\alpha}x^{2-\alpha}\ell(x)
\end{equation*}
and
\begin{equation*}   \textstyle
R(x) \sim \frac{\alpha}{\alpha-1} x(1-F(x)) \sim \frac{\alpha}{\alpha-1}x^{1-\alpha}\ell(x)
\end{equation*}
as $x \to \infty$.
For any $A>0$
and $\delta>0$ satisfying \eqref{eq:m_theta(alpha+-delta)},
there exists $x_0>0$ such that, whenever $x \geq x_0$, we have
\begin{align}
x^{\alpha+\delta}(1-F(x))   &\geq   1/A;    \label{eq:A1}   \\
x^{\alpha-2+\delta}T(x) &\geq   1/A ;   \label{eq:A2}   \\
x^{\alpha-1+\delta}R(x) &\geq   1/A ;   \label{eq:A3}   \\
T(x)    &\leq \textstyle
\big(A +\frac{\alpha}{2-\alpha}\big)x^2(1-F(x))\defeq B_1x^2(1-F(x));   \label{eq:A4}   \\
R(x)    &\leq \textstyle
\big(A +\frac{\alpha}{\alpha-1}\big)x(1-F(x))\defeq B_2x(1-F(x)).   \label{eq:A5}
\end{align}
Also, $x_0$ can be chosen so large that (with the same $\delta$ as before)
whenever $x \wedge (u x) \geq x_0$, we have
\begin{align}
\tfrac{1-F(ux)}{1-F(x)} &\leq A  (u^{-\alpha+\delta} \vee u^{-\alpha-\delta});  \label{eq:A6}   \\
\tfrac{T(ux)}{T(x)}     &\leq A  (u^{2-\alpha+\delta} \vee u^{2-\alpha-\delta});    \label{eq:A7}   \\
\tfrac{R(ux)}{R(x)}     &\leq A (u^{1-\alpha+\delta} \vee u^{1-\alpha-\delta}). \label{eq:A8}
\end{align}
Inequalities \eqref{eq:A6} through \eqref{eq:A8} follow from
Potter's bound (Theorem 1.5.6(iii) in \cite{Bingham+Goldie+Teugels:1989}).
While constructing bounds for $I_1$, $I_2$ and $I_3$ below we tacitly assume that $x\geq x_0$.  \smallskip

\noindent {\sc A bound for $I_1$}.
Write
\begin{align*}
\frac{I_1(n,x)}{1-F(x)}
&= \frac{\Prob_n\big(\sup_{|u| \geq n}|Y_{u,a}Q_u| >x\big)}{1-F(x)}
\leq \E_n\bigg[ \sum_{|u| \geq n} \frac{\Prob_{|u|}(|Y_{u,a}Q_u| >x)}{1-F(x)}\bigg]\\
&= \E_n\bigg[\sum_{|u| \geq n} \frac{1-F(x/Y_{u,a})}{1-F(x)}\bigg]
=\E_n\bigg[\sum_{|u| \geq n} \frac{1-F(x/Y_{u,a})}{1-F(x)} \1_{\{|Y_{u,a}|>x/x_0\}}\bigg]\\
&\hphantom{~=~}+\E_n\bigg[\sum_{|u| \geq n} \frac{1-F(x/Y_{u,a})}{1-F(x)} \1_{\{|Y_{u,a}| \leq x/x_0\}}\bigg] \eqdef I_{11}(n,x)+I_{12}(n,x).
\end{align*}
For $|u|\ge n$, we have
\begin{equation*}
|Y_{u,a}|^{\alpha+\delta}
\geq |Y_{u,a}|^{\alpha+\delta}\1_{\{|Y_{u,a}|>x/x_0\}}
\geq (x/x_0)^{\alpha+\delta}\1_{\{|Y_{u,a}|>x/x_0\}}.
\end{equation*}
From this, we conclude that
\begin{equation*}
I_{11}(n,x)
\leq x_0^{\alpha+\delta} \E_n\bigg[\frac{\sum_{|u| \geq n}|Y_{u,a}|^{\alpha+\delta}}{x^{\alpha+\delta}(1-F(x))}\bigg]
\leq A x_0^{\alpha+\delta}\E_n \bigg[\sum_{|u| \geq n}|Y_{u,a}|^{\alpha+\delta}\bigg]
\end{equation*}
by \eqref{eq:A1}.
Further, we obtain with the help of \eqref{eq:A6}
\begin{align*}
I_{12}(n,x)
&\leq   A \E_n\bigg[\sum_{|u| \geq n}\left(|Y_{u,a}|^{\alpha-\delta}\vee |Y_{u,a}|^{\alpha+\delta}\right)\bigg]
\leq    A \E_n\bigg[\sum_{j \geq n}|a_j|^{\alpha-\delta} \Xi_j \bigg].
\end{align*}

\noindent {\sc A bound for $I_2$}.
By Markov's inequality
\begin{align*}
(x/2)^2I_2(n,x)
&\leq
\E_n \bigg[\bigg(\sum_{|u| \geq n}\big(Y_{u,a} Q_u \1_{\{|Y_{u,a}Q_u| \leq x\}}-\E_{|u|}[Y_{u,a} Q_u \1_{\{|Y_{u,a}Q_u| \leq x\}}]\big)\bigg)^2\bigg]   \\
&\leq
\E_n\bigg[\sum_{|u| \geq n}(Y_{u,a})^2Q_u^2\1_{\{|Y_{u,a}Q_u|\leq x\}}\bigg],
\end{align*}
as the expectations of the cross terms vanish.
By virtue of \eqref{eq:A4} we get
\begin{align*}
&\frac{I_2(n,x)}{4(1-F(x))}\leq
\E_n\bigg[\sum_{|u| \geq n}\frac{Y_{u,a}^2\int_0^{x/|Y_{u,a}|} y^2 \, \dF(y)}{x^2(1-F(x))}\bigg]
\leq B_1\E_n\bigg[\sum_{|u| \geq n} \frac{Y_{u,a}^2T(x/|Y_{u,a}|)}{T(x)}\bigg]  \\
&\qquad= B_1\E_n\bigg[\sum_{|u| \geq n} \frac{Y_{u,a}^2T(\frac{x}{|Y_{u,a}|})}{T(x)} \1_{\{|Y_{u,a}|> \frac{x}{x_0}\}}
+\sum_{|u| \geq n} \frac{Y_{u,a}^2T(\frac{x}{|Y_{u,a}|})}{T(x)} \1_{\{|Y_{u,a}| \leq \frac{x}{x_0}\}}\bigg] \\
&\qquad\eqdef   B_1(I_{21}(n,x)+I_{22}(n,x)).
\end{align*}
We use \eqref{eq:A2} and the trivial inequality $T(x)\leq x^2$ for
$x\geq 0$ to obtain
\begin{align*} I_{21}&(n,x)=
\E_n\bigg[\sum_{|u| \geq n}\frac{|Y_{u,a}|^{\alpha+\delta}(x/|Y_{u,a}|)^{\alpha-2+\delta}T(x/|Y_{u,a}|)}{x^{\alpha-2+\delta}
T(x)}\1_{\{Y_{u,a}> x/x_0\}}\bigg]\\&\leq A \max_{y\in
[0,x_0]}(y^{\alpha-2+\delta}T(y))\E_n \bigg[\sum_{|u| \geq n}|Y_{u,a}|^{\alpha+\delta}\bigg]\leq
A  x_0^{\alpha+\delta}\E_n\bigg[\sum_{|u| \geq n}|Y_{u,a}|^{\alpha+\delta}\bigg].
\end{align*}
Further, as a consequence of \eqref{eq:A7},
\begin{align*}I_{22}(n,x)&\leq A  \E_n \bigg[\sum_{|u| \geq n}Y_{u,a}^2(|Y_{u,a}|^{\alpha-2-\delta}\vee
|Y_{u,a}|^{\alpha-2+\delta})\bigg]
\leq A \E_n\bigg[\sum_{j \geq n}|a_j|^{\alpha-\delta} \Xi_j \bigg].
\end{align*}

\noindent {\sc A bound for $I_3$}.
We first observe that for $|u|\ge n$
\begin{align*}
\E_n[Y_{u,a}Q_u\1_{\{|Y_{u,a}Q_u|\leq x\}}]
&= \E_n \bigg[Y_{u,a} \int_{\{|y|\leq x/|Y_{u,a}|\}} y \, \dProb(Q \leq y)\bigg]\\
&=-\E_n\bigg[Y_{u,a}\int_{\{|y|>x/|Y_{u,a}|\}}y \, \dProb(Q\leq y)\bigg]
\end{align*}
whence
\begin{equation*}
|\E_n[Y_{u,a}Q_u\1_{\{|Y_{u,a}Q_u|\leq x\}}]|
\leq \E_n \bigg[|Y_{u,a}|\! \int_{(x/|Y_{u,a}|,\infty)} \!\!\!\!\! y \, \dF(y) \bigg]
=\E_n \Big[|Y_{u,a}| R(x/|Y_{u,a}|)\Big].
\end{equation*}
Consequently, by Markov's inequality and \eqref{eq:A5},
\begin{align*}
\frac{I_3(n,x)}{2(1-F(x))}
&\leq   \E_n \bigg[\sum_{|u| \geq n}\frac{|Y_{u,a}| R(x/|Y_{u,a}|)}{x(1-F(x))}\bigg]
\leq B_2\E_n \bigg[\sum_{|u| \geq n}\frac{|Y_{u,a}|R(x/|Y_{u,a}|)}{R(x)}\bigg]  \\
&=B_2\E_n\bigg[\sum_{|u| \geq n} \frac{|Y_{u,a}|R(x/|Y_{u,a}|)}{R(x)} \1_{\{|Y_{u,a}|> x/x_0\}}\\
&\phantom{=B_2\E_n\bigg[}+\sum_{|u| \geq n} \frac{|Y_{u,a}|R(x/|Y_{u,a}|)}{R(x)} \1_{\{|Y_{u,a}|\leq x/x_0\}}\bigg]\\
&\eqdef B_2(I_{31}(n,x)+I_{32}(n,x)).
\end{align*}
Using \eqref{eq:A3} and the fact that $R(x)$ is nonincreasing we
conclude that
\begin{align*}
I_{31}(n,x)
&= \E_n\bigg[ \sum_{|u| \geq n}\frac{|Y_{u,a}|^{\alpha+\delta}_u(x/|Y_{u,a}|)^{\alpha-1+\delta}R(x/|Y_{u,a}|)}{x^{\alpha-1+\delta} R(x)}\1_{\{|Y_{u,a}|> x/x_0\}}\bigg] \\
&\leq A \E_n \bigg[\max_{y\in [0,x_0]}(y^{\alpha-1+\delta}R(y))\sum_{|u| \geq n}|Y_{u,a}|^{\alpha+\delta}\bigg] \\
&\leq A  \, \E[|W_1-1|] x_0^{\alpha-1+\delta} \E_n\bigg[\sum_{|u| \geq n}|Y_{u,a}|^{\alpha+\delta}\bigg].
\end{align*}
Finally, by \eqref{eq:A8},
\begin{align*}
I_{32}(n,x) &\leq A  \E_n \bigg[\sum_{|u| \geq n}|Y_{u,a}|\Big(|Y_{u,a}|^{\alpha-1-\delta}\vee |Y_{u,a}|^{\alpha-1+\delta}\Big)\bigg]
\\&\leq A \E_n \bigg[\sum_{|u| \geq n}|Y_{u,a}|^{\alpha-\delta}+\sum_{|u| \geq n}|Y_{u,a}|^{\alpha+\delta}\bigg].
\end{align*}
The preceding inequalities imply \eqref{eq:basic tail estimate <grad W,a> under P_n uniform} with $\Xi_k$ as defined in \eqref{eq:Xi_n}.\smallskip

Now some preparatory work has to be done for the next part of the proof. Since $\E[W_1-1]=0$, $\Prob(|W_1-1| > x) \sim x^{-\alpha} \ell(x)$ by \eqref{eq:W_1(theta) reg varying tail},
$\Prob(W_1-1<-x)=0$ for $x>1$ and $\sum_{|u|=n} Y_u^{\alpha-\delta}<\infty$ a.\,s.\
for any $n \in \N$ as a consequence of \eqref{eq:sum_u Y_u^p < infty},
Lemma A.3 in \cite{Mikosch+Samorodnitsky:2000} or Theorem 2.2 in \cite{Kokoszka+Taqqu:1996} give that,
as $x \to \infty$,
\begin{equation}    \label{eq:basic tail estimate W_n+1-W_n under P_n}  \textstyle
\Prob_n(W_{n+1}-W_n>x) \sim \sum_{|u|=n}Y_u^\alpha (1-F(x))
\quad   \text{a.\,s.}
\end{equation}
This in combination with \eqref{eq:basic tail estimate <grad W,a> under P_n uniform} and Lebesgue's dominated convergence theorem enables us to conclude that, as $x \to \infty$,
\begin{equation}    \label{eq:basic tail estimate W_n+1-W_n under P}
\Prob(W_{n+1}-W_n>x)\sim    m_\theta(\alpha)^n (1-F(x)),
\quad n \in \N_0.
\end{equation}
Alternatively, using an inductive argument relation \eqref{eq:basic tail estimate W_n+1-W_n under P} can be deduced from
Theorem 2.1 in \cite{Olvera-Cravioto:2012} and the remark following Theorem 2.2 in \cite{Olvera-Cravioto:2012}.

We are ready to finish the proof of \eqref{eq:tail <grad W,a>}. We claim that
\begin{equation}
\Prob\bigg(\sum_{j=0}^k a_j (W_{j+1}(\theta)-W_j(\theta))>x\bigg)\sim
\sum_{j=0}^k \kappa^j(a_j^+)^{\alpha}\, \Prob(W_1(\theta)>x)        \label{eq:tail <grad W,a>red}
\end{equation}
for $k\in\N_0$. This will be proved by induction on $k$.

\noindent
For $k=0$, \eqref{eq:tail <grad W,a>red} is \eqref{eq:W_1(theta) reg varying tail}, which is an assumption.

\noindent
Now suppose that \eqref{eq:tail <grad W,a>red} holds for fixed $k\in\N$.
Then, for $x>0$ and $\rho \in (0,1)$,
\begin{align*}
\Prob&\bigg(\sum_{j=0}^{k+1}a_j(W_{j+1}-W_j)>x\bigg)    \\
&\leq
\Prob\bigg(\sum_{j=0}^k a_j(W_{j+1}-W_j)>(1-\rho)x\bigg)+\Prob\big(a_{k+1}(W_{k+2}-W_{k+1})>(1-\rho)x\big) \\
&\hphantom{\leq~}+\Prob\bigg(\sum_{j=0}^{k}a_j(W_{j+1}-W_j)>\rho x, a_{k+1}(W_{k+2}-W_{k+1})>\rho x\bigg)   \\
&=
\Prob\bigg(\sum_{j=0}^{k}a_j(W_{j+1}-W_j)>(1-\rho)x\bigg)+\Prob\big(a_{k+1}(W_{k+2}-W_{k+1})>(1-\rho)x\big) \\
&\hphantom{\leq~}+\E\Big[ \1_{\{\sum_{j=0}^{k}a_j(W_{j+1}-W_j)>\rho x\}}\Prob_{k+1}\big(a_{k+1}(W_{k+2}-W_{k+1})>\rho x\big)\Big],
\end{align*}
where we used the fact that the variable $\sum_{j=0}^{k}a_j(W_{j+1}-W_j)$ is $\F_{k+1}$-measurable.
Set $\zeta_1 \defeq 0$, $\zeta_2 \defeq ((a_{k+1}^+)/\rho)^{\alpha} \sum_{|u|=k+1}Y_u^\alpha$ and,
for $x>0$,
\begin{equation*}
\zeta_1(x)\defeq \1_{\{\sum_{j=0}^{k}a_j(W_{j+1}-W_j)>\rho x\}}\frac{\Prob_{k+1}(a_{k+1}(W_{k+2}-W_{k+1})>\rho x)}{1-F(x)},
\end{equation*}
\begin{equation*}
\zeta_2(x)\defeq \frac{\Prob_{k+1}(a_{k+1}(W_{k+2}-W_{k+1})>\rho x)}{1-F(x)}.
\end{equation*}
In view of \eqref{eq:basic tail estimate W_n+1-W_n under P_n}, we have
$\lim_{x\to\infty} \zeta_1(x)=\zeta_1$ a.\,s.\ and $\lim_{x\to\infty} \zeta_2(x)=\zeta_2$ a.\,s.
Further, $\lim_{x\to\infty} \E[\zeta_2(x)]=\E[\zeta_2]$ by \eqref{eq:basic tail estimate W_n+1-W_n under P}.
Since, for $x>0$, we have $0\leq \zeta_1(x)\leq \zeta_2(x)$ a.\,s.,
we can invoke Pratt's lemma \cite{Pratt:1960} to get $\lim_{x \to \infty} \E[\zeta_1(x)]=\E[\zeta_1]$. Hence,
\begin{equation}    \label{eq:3rd term}
\lim_{x \to \infty} \frac{\Prob\big(\sum_{i=0}^{k}a_i(W_{i+1}-W_i)>\rho x, a_{k+1}(W_{k+2}-W_{k+1})>\rho x\big)}{1-F(x)}=0.
\end{equation}
By the induction hypothesis, \eqref{eq:basic tail estimate W_n+1-W_n under P} and \eqref{eq:3rd term}
\begin{equation*}
\limsup_{x \to \infty}
\frac{\Prob\big(\sum_{j=0}^{k+1}a_j(W_{j+1}-W_j)>x\big)}{1-F(x)}\leq (1-\rho)^{-\alpha}\sum_{j=0}^{k+1} m_\theta(\alpha)^{j}(a_{j}^+)^{\alpha}.
\end{equation*}
Letting $\rho \downarrow 0$ yields
\begin{equation}    \label{eq:limsup tail under P}
\limsup_{x\to\infty}
\frac{\Prob\big(\sum_{j=0}^{k+1}a_j(W_{j+1}-W_j)>x\big)}{1-F(x)} \leq \sum_{j=0}^{k+1} m_\theta(\alpha)^{j}(a_{j}^+)^{\alpha}.
\end{equation}
We now derive the corresponding inequality for the limit inferior.
To this end, for $x>0$ and $\rho>0$, we write
\begin{align*}
\Prob\bigg(&\sum_{j=0}^{k+1}a_j(W_{j+1}-W_j)>x\bigg)\\
&\geq   \Prob\bigg(\sum_{j=0}^{k}a_j(W_{j+1}-W_j)>(1+\rho)x, |a_{k+1}(W_{k+2}-W_{k+1})|\leq \rho x\bigg)    \\
&\hphantom{\geq~}
+\Prob\Big(a_{k+1}(W_{k+2}-W_{k+1})>(1+\rho)x, \Big|\sum_{j=0}^{k}a_j(W_{j+1}-W_j)\Big|\leq \rho x\Big) \\
&=
\Prob\bigg(\sum_{j=0}^{k}a_j(W_{j+1}-W_j)>(1+\rho) x\bigg)  \\
&\hphantom{=~}
- \Prob\bigg(\sum_{j=0}^{k}a_j(W_{j+1}-W_j)>(1+\rho)x, |a_{k+1}(W_{k+2}-W_{k+1})|>\rho x\bigg)  \\
&\hphantom{=~}+
\Prob\big(a_{k+1}(W_{k+2}-W_{k+1})>(1+\rho) x\big)  \\
&\hphantom{=~}-
\Prob\bigg(a_{k+1}(W_{k+2}-W_{k+1})>(1+\rho) x, |\sum_{j=0}^{k}a_j(W_{j+1}-W_j)|>\rho x\bigg).
\end{align*}
The argument that led to \eqref{eq:3rd term} applies here as well.
It gives
\begin{equation}    \label{eq:Pratt 2}
\lim_{x \to \infty} \frac{\Prob\big(\sum_{j=0}^{k}a_j(W_{j+1}\!-\!W_j)>(1+\rho) x, |a_{k+1}(W_{k+2}\!-\!W_{k+1})|>\rho x\big)}{1-F(x)}=0
\end{equation}
and
\begin{equation}    \label{eq:Pratt 3}
\lim_{x\to\infty} \frac{\Prob\big(a_{k+1}(W_{k+2}\!-\!W_{k+1})>(1+\rho) x,
|\sum_{j=0}^{k}a_j(W_{j+1}\!-\!W_j)|>\rho x\big)}{1-F(x)}=0.
\end{equation}
By the induction hypothesis, \eqref{eq:basic tail estimate W_n+1-W_n under P}, \eqref{eq:Pratt 2} and
\eqref{eq:Pratt 3}
\begin{align*}
\liminf_{x \to \infty}
\frac{\Prob\big(\sum_{j=0}^{k+1}a_j(W_{j+1}-W_j)>x\big)}{1-F(x)}    
\geq (1+\rho)^{-\alpha}\sum_{j=0}^{k+1} m_\theta(\alpha)^{j}(a_j^+)^{\alpha}.
\end{align*}
Upon letting $\rho \downarrow 0$, we obtain
\begin{equation}    \label{eq:liminf tail under P}
\liminf_{x \to \infty}
\frac{\Prob\big(\sum_{j=0}^{k+1}a_j(W_{j+1}-W_j)>x\big)}{1-F(x)}\geq \sum_{j=0}^{k+1} m_\theta(\alpha)^{j}(a_j^+)^{\alpha}.
\end{equation}
Combining \eqref{eq:limsup tail under P} and \eqref{eq:liminf tail under P} gives \eqref{eq:tail <grad W,a>red}
for $k+1$, thereby proving \eqref{eq:tail <grad W,a>red} in general.  

To check \eqref{eq:tail <grad W,a>} we fix $k \in \N_0$, $x>0$ and $\rho\in (0,1)$, and write
\begin{align}
\Prob&\bigg(\sum_{j=0}^k a_j(W_{j+1}-W_j)>(1+\rho)x\bigg)-\Prob\bigg(\Big|\sum_{j\geq k+1} a_j(W_{j+1}-W_j)\Big|>\rho x\bigg)  \label{ineq:1st for W-W_n}  \\
&\leq \Prob\bigg(\sum_{j\geq 0} a_j(W_{j+1}-W_j)>x\bigg)    \notag  \\
&\leq \Prob\bigg(\sum_{j=0}^k a_j(W_{j+1}-W_j)>(1-\rho)x\bigg) + \Prob\bigg(\Big|\sum_{j\geq k+1} a_j(W_{j+1}-W_j)\Big|>\rho x\bigg).  \notag
\end{align}
From \eqref{ineq:1st for W-W_n}, \eqref{eq:tail <grad W,a>red} and \eqref{eq:basic tail estimate <grad W,a> under P_n uniform}, we infer
\begin{align*}
(1+\rho)^{-\alpha}&\sum_{j=0}^k m_\theta(\alpha)^{j}(a_j^+)^{\alpha} - C \rho^{-\alpha} \sum_{j \geq k+1}|a_j|^{\alpha-\delta} \E[\Xi_j]  \\
&\leq
\liminf_{x\to\infty}\frac{\Prob(\sum_{j\geq 0} a_j(W_{j+1}-W_j)>x)}{1-F(x)} \\
&\leq
\limsup_{x\to\infty}\frac{\Prob(\sum_{j\geq 0} a_j(W_{j+1}-W_j)>x)}{1-F(x)} \\
&\leq
(1-\rho)^{-\alpha}\sum_{j=0}^k m_\theta(\alpha)^{j}(a_j^+)^{\alpha}+ C \rho^{-\alpha} \sum_{j \geq k+1}|a_j|^{\alpha-\delta} \E[\Xi_j]
\end{align*}
Letting $k \to \infty$ and then $\rho \downarrow 0$,
we arrive at \eqref{eq:tail <grad W,a>}. The proof of \eqref{eq:tail <grad W,a->} is analogous, hence omitted.

A perusal of the proof above reveals that the need for condition
\eqref{eq:m((alpha+epsilon)theta)<infty} is only motivated by the use of Potter's bound, see
\eqref{eq:A6}, \eqref{eq:A7} and \eqref{eq:A8}. If
$\lim_{x\to\infty}\,\ell(x)=c$, that is, condition \eqref{eq:tail assumption W_1}
holds, inequality \eqref{eq:A6} can be replaced by the following:
for any $A >1$ there exists $x_0>0$ such that whenever $x\geq x_0$
and $ux \geq x_0$,
\begin{equation*}
\tfrac{1-F(ux)}{1-F(x)}\leq A  u^{-\alpha},
\end{equation*}
likewise for \eqref{eq:A7} and \eqref{eq:A8} (with the same $x_0$ as $x_0$ can be increased if necessary).
This shows that condition
\eqref{eq:m((alpha+epsilon)theta)<infty} is no longer needed,
\eqref{eq:contraction assumption} being sufficient. The proof of Theorem \ref{Thm:tail of W(theta)} is complete.

\section{Proof of Theorem \ref{Thm:stable fluctuations of Biggins's martingale}}    \label{sec:stable fluctuations of Biggins's martingale}

Our proof of Theorem \ref{Thm:stable fluctuations of Biggins's martingale} is essentially based on the
following result in combination with Theorem \ref{Thm:tail of W(theta)}.

\begin{Lemma}   \label{Lem:Theta_n convergence}
Let $V=f((\cZ(u))_{u \in \I})$ for a measurable function $f$
such that $\E[V]=0$ and
\begin{equation}    \label{eq:tails of V}
\Prob(V>x)~\sim~ c_1x^{-\alpha}\quad\text{and}\quad
\Prob(-V>x)~\sim~ c_2x^{-\alpha},\quad x\to\infty
\end{equation}
for some $\alpha\in (1,2)$ and finite $c_1,c_2 \geq 0$ with $c_1+c_2>0$.
Further, suppose that $m(\alpha\theta)<\infty$ (\eqref{eq:contraction assumption} is not required).
For $n \in \N$, set
\begin{equation}    \label{xn}
\Theta_n\defeq m(\alpha\theta)^{-n/\alpha}\sum_{|u|=n}e^{-\theta S(u)}V^{(u)},
\end{equation}
where $V^{(u)} = f((\cZ(uv))_{v \in \I})$ for $u \in \I$.
Then, for $t \in \R$,
\begin{align}
&\lim_{n \to \infty} \,\E\big[\exp(\imag t \Theta_n)\big]   \notag  \\
&~=
\E\bigg[\!\exp\!\bigg(\frac{\Gamma(2\!-\!\alpha)}{\alpha\!-\!1}W(\alpha\theta)|t|^\alpha
\Big(\!(c_1\!+\!c_2)\cos \! \Big(\frac{\pi\alpha}{2}\Big)-\imag (c_1\!-\!c_2)\sin \! \Big(\frac{\pi\alpha}{2}\Big)\sign(t) \Big)\!\bigg)\bigg]\!.
\label{eq:convergence to stable}
\end{align}
\end{Lemma}
\begin{proof}
Since, conditionally given $\F_n$, $\Theta_n$ is a weighted sum of i.i.d.\ random variables,
\eqref{eq:convergence to stable} follows from the classical limit theory for triangular arrays.

Suppose we can check that, for every $x>0$,
\begin{align}
L(x) &\defeq -\lim_{n\to\infty}\,\sum_{|u|=n} \Prob_n\bigg(\frac{e^{-\theta S(u)}V^{(u)}}{m(\alpha\theta)^{n/\alpha}} > x\bigg)
=-c_1x^{-\alpha}W(\alpha\theta)\quad\text{a.\,s.;}      \label{eq:L(x)>0}   \\
L(-x) &\defeq \lim_{n\to\infty}\,\sum_{|u|=n} \Prob_n
\bigg(\frac{e^{-\theta S(u)} V^{(u)}}{m(\alpha\theta)^{n/\alpha}} \leq -x\bigg)
=c_2 x^{-\alpha}W(\alpha\theta)\quad\text{a.\,s.};  \label{eq:L(x)<0}
\end{align}
\begin{equation}    \label{eq:sigma}
\sigma^2
\defeq \lim_{\varepsilon \downarrow 0} \lim_{n\to\infty}\,
\sum_{|u|=n} \Var_n\bigg[\frac{e^{-\theta S(u)} V^{(u)}}{m(\alpha\theta)^{n/\alpha}} \1_{\big\{\frac{e^{-\theta
S(u)}|V^{(u)}|}{m(\alpha\theta)^{n/\alpha}} \leq \varepsilon\big\}}\bigg]=0 \quad   \text{a.\,s.}
\end{equation}
and
\begin{align}   \label{eq:a}
a_0(\tau)
&\defeq
\lim_{n\to\infty}\,\sum_{|u|=n} \E_n\bigg[\frac{e^{-\theta S(u)}V^{(u)}}{m(\alpha\theta)^{n/\alpha}}
\1_{\{|e^{-\theta S(u)} V^{(u)}| \leq \tau m(\alpha\theta)^{n/\alpha}\}}\bigg]  \notag  \\
&= -\tau^{1-\alpha} \frac{\alpha (c_1-c_2)}{\alpha-1} W(\alpha\theta) \quad \text{a.\,s.}
\end{align}
for each $\tau>0$. Then, according to Theorem 1 on p.~116 in \cite{Gnedenko+Kolmogorov:1968},
\begin{align}
&\lim_{n\to\infty} \E_n\big[\imag t \Theta_n\big]   \notag  \\
&=\exp\bigg(\imag at-\frac{\sigma^2t^2}{2}+\int_{\R \setminus \{0\}}\bigg(e^{\imag tx}-1-\frac{\imag tx}{1+x^2}\bigg) \, \dL(x)\bigg)\notag \\
&=
\exp\bigg(-\alpha c_2W(\alpha\theta)\bigg(\frac{\imag \pi t}{2\cos(\frac{\pi\alpha}{2})}-\int_{-\infty}^0
\bigg(e^{\imag tx}-1-\frac{\imag tx}{1+x^2}\bigg)|x|^{-\alpha-1}\, \dx \bigg)\bigg)\notag   \\
&\hphantom{=}~\cdot\exp\bigg(\alpha c_1W(\alpha\theta) \bigg(\frac{\imag \pi t}{2\cos(\frac{\pi\alpha}{2})}
+\int_0^\infty \bigg(e^{\imag tx}-1-\frac{\imag tx}{1+x^2}\bigg)x^{-\alpha-1}\, \dx \bigg)\bigg)    \text{ a.\,s.}
\label{eq:Kolmogorov's limit relation}
\end{align}
for $t\in\R$. Here,
\begin{equation*}
a \defeq a_0(\tau)-\int_{[-\tau,\,\tau]}\frac{x^3}{1+x^2} \, \dL(x)
+\int_{\R \setminus [-\tau,\,\tau]}\frac{x}{1+x^2} \, \dL(x)=\frac{\alpha (c_1-c_2)\pi W(\alpha\theta)}{2\cos(\frac{\pi\alpha}{2})}
\end{equation*}
as a consequence of
\begin{equation*}
\int_0^\tau \! \frac{x^{2-\alpha}}{1\!+\!x^2} \, \dx  - \int_\tau^\infty \!\! \frac{x^{-\alpha}}{1\!+\!x^2}\, \dx
=\int_0^\infty \!\! \frac{x^{2-\alpha}}{1\!+\!x^2}\, \dx - \int_\tau^\infty \!\! x^{-\alpha}\, \dx
=-\frac{\pi}{2\cos(\frac{\pi\alpha}{2})}-\frac{\tau^{1-\alpha}}{\alpha-1}.
\end{equation*}
The last equality follows from
\begin{align}
\int_0^\infty \!\! \frac{x^{2-\alpha}}{1+x^2}\, \dx
&= \frac{1}{2} \int_0^1 \! x^{(1-\alpha)/2}(1-x)^{-(3-\alpha)/2}\, \dx
=\frac{1}{2} \Gamma\Big(\frac{3-\alpha}{2}\Big)\Gamma\Big(1-\frac{3-\alpha}{2}\Big)\notag   \\
&=\frac{\pi}{2\sin\big(\frac{\pi(3-\alpha)}{2}\big)}=-\frac{\pi}{2\cos(\frac{\pi\alpha}{2})}.   \label{eq:integral identity}
\end{align}
In view of \eqref{eq:integral identity} the right-hand side of \eqref{eq:Kolmogorov's limit relation}
equals
\begin{align*}
&\exp\bigg(\alpha c_2W(\alpha\theta)\int_{-\infty}^0 \big(e^{\imag tx}-1-\imag tx\big)|x|^{-\alpha-1}\, \dx \bigg)\bigg)    \\
&~\cdot \exp\bigg(\alpha c_1W(\alpha\theta)\int_0^\infty \big(e^{\imag tx}-1-\imag tx\big) x^{-\alpha-1}\, \dx \bigg)\bigg) \\
&=
\exp\bigg(\frac{\Gamma(2-\alpha)}{\alpha-1}c_2W(\alpha\theta)|t|^\alpha\big(\cos(\tfrac{\pi\alpha}{2})+\imag \sin(\tfrac{\pi\alpha}{2})\,\sign(t)\big)\bigg)    \\
&\hphantom{=}~\cdot
\exp\bigg(\frac{\Gamma(2-\alpha)}{\alpha-1}c_1W(\alpha\theta)|t|^\alpha\big(\cos(\tfrac{\pi\alpha}{2})-\imag \sin(\tfrac{\pi\alpha}{2})\,\sign(t)\big)\bigg)    \\
&= \exp\bigg(\frac{\Gamma(2-\alpha)}{\alpha-1}W(\alpha\theta)|t|^\alpha \big((c_1+c_2)\cos(\tfrac{\pi\alpha}{2})-\imag (c_1-c_2)\sin(\tfrac{\pi\alpha}{2})\,\sign(t)\big)\bigg)
\end{align*}
having utilized the first formula given on p.\;170 in \cite{Gnedenko+Kolmogorov:1968} for the penultimate equality.
Now \eqref{eq:convergence to stable} is secured by \eqref{eq:Kolmogorov's limit relation}, the last displayed formula
and Lebesgue's dominated convergence theorem.

Next, we are passing to the proofs of \eqref{eq:L(x)>0} through \eqref{eq:a}.

\noindent {\sc Proofs of \eqref{eq:L(x)>0} and \eqref{eq:L(x)<0}}.
We start by recalling that, by Theorem 3 in \cite{Biggins:1998},
\begin{equation}    \label{eq:Biggins:1998 for Y_u^beta}
\lim_{n\to\infty} \,\sup_{|u|=n} \frac{e^{-\theta S(u)}}{m(\alpha\theta)^{n/\alpha}}=0 \quad \text{a.\,s.}
\end{equation}
Using this in combination with \eqref{eq:tails of V} gives, for any $x > 0$,
\begin{align*}
\sum_{|u|=n} \Prob_n \Big(\frac{e^{-\theta S(u)} V^{(u)}}{m(\alpha\theta)^{n/\alpha}} > x\Big)
&\sim \sum_{|u|=n} c_1\big(xe^{\theta S(u)}m(\alpha\theta)^{n/\alpha}\big)^{-\alpha} \notag  \\
&= c_1x^{-\alpha} W_n(\alpha\theta)~ \to~ c_1x^{-\alpha}
W(\alpha\theta)\quad \text{a.\,s.} 
\end{align*}
as $n \to \infty$. This proves \eqref{eq:L(x)>0}. The proof of \eqref{eq:L(x)<0} is analogous.

\noindent {\sc Proof of \eqref{eq:sigma}}.
For $\varepsilon > 0$,
\begin{align*}
\sum_{|u|=n} & \Var_n \! \bigg[\frac{e^{-\theta S(u)}V^{(u)}}{m(\alpha\theta)^{n/\alpha}}
\1_{\big\{\frac{e^{-\theta S(u)}|V^{(u)}|}{m(\alpha\theta)^{n/\alpha}} \leq \varepsilon\big\}} \bigg]    \\
&=
\sum_{|u|=n} \frac{e^{-2\theta S(u)}}{m(\alpha\theta)^{2n/\alpha}} \,
\E_n\big[(V^{(u)})^2 \1_{\{|V^{(u)}| \leq e^{\theta S(u)} m(\alpha\theta)^{n/\alpha} \varepsilon\}} \big]   \\
&=~ \sum_{|u|=n} \frac{e^{-2\theta S(u)}}{m(\alpha\theta)^{2n/\alpha}}
\int_{[0, e^{\theta S(u) m(\alpha\theta)^{n/\alpha}}\varepsilon]} y^2 \, \dProb(|V|\leq y\}.
\end{align*}
Observe that \eqref{eq:tails of V} entails
\begin{equation*}
\Prob(|V|>x) \sim (c_1+c_2)x^{-\alpha} \quad   \text{as } x \to \infty.
\end{equation*}
Integration by parts thus leads to
\begin{equation*}   \textstyle
\int_{[0,\,x]} y^2 \, \dProb(|V|\leq y)
\sim \frac{\alpha
(c_1+c_2)}{2-\alpha} x^{2-\alpha}   \quad   \text{as } x \to \infty.
\end{equation*}
Using this and \eqref{eq:Biggins:1998 for Y_u^beta}, we conclude
that, as $n \to \infty$,
\begin{align*}
\sum_{|u|=n} & \Var_n \! \bigg[\frac{e^{-\theta S(u)}V^{(u)}}{m(\alpha\theta)^{n/\alpha}}
\1_{\{|V^{(u)}| \leq e^{\theta S(u)} m(\alpha\theta)^{n/\alpha} \varepsilon\}} \bigg]    \\
&=~  \sum_{|u|=n} \frac{e^{-2\theta S(u)}}{m(\alpha\theta)^{2n/\alpha}}
\int_{[0, e^{\theta S(u) m(\alpha\theta)^{n/\alpha}}\varepsilon]} y^2 \, \dProb(|V|\leq y) \\
&\sim~
\sum_{|u|=n} \frac{e^{-2\theta S(u)}}{m(\alpha\theta)^{2n/\alpha}}
\frac{\alpha (c_1+c_2)}{2-\alpha} \big(e^{\theta S(u)} m(\alpha\theta)^{n/\alpha} \varepsilon \big)^{2-\alpha}  \\
&=~ \varepsilon^{2-\alpha} \frac{\alpha (c_1+c_2)}{2-\alpha} \sum_{|u|=n} \frac{e^{-\alpha\theta S(u)}}{m(\alpha\theta)^{n}}  \\
&=~ \varepsilon^{2-\alpha} \frac{\alpha (c_1+c_2)}{2-\alpha} W_n(\alpha\theta)
~\to~ \varepsilon^{2-\alpha} \frac{\alpha (c_1+c_2)}{2-\alpha} W(\alpha\theta) \quad \text{a.\,s.}
\end{align*}
This last expression vanishes as $\varepsilon \downarrow 0$
which proves \eqref{eq:sigma}.

\noindent {\sc Proof of \eqref{eq:a}}.
For every $\tau > 0$, since $\E[V]=0$, we have
\begin{align*}
\sum_{|u|=n} & \E_n\bigg[\frac{e^{-\theta S(u)}V^{(u)}}{m(\alpha\theta)^{n/\alpha}}
\1_{\big\{\big|\frac{e^{-\theta S(u)}V^{(u)}}{m(\alpha\theta)^{n/\alpha}}\big| \leq \tau\big\}} \bigg]    \\
&=~ -\sum_{|u|=n} \E_n\bigg[\frac{e^{-\theta S(u)}V^{(u)}}{m(\alpha\theta)^{n/\alpha}}
\1_{\big\{\big|\frac{e^{-\theta S(u)}V^{(u)}}{m(\alpha\theta)^{n/\alpha}}\big| > \tau\big\}} \bigg]    \\
&=~ -\sum_{|u|=n} \frac{e^{-\theta S(u)}}{m(\alpha\theta)^{n/\alpha}}
\E_n\bigg[V^{(u)}\1_{\big\{|V^{(u)}| > \tau \frac{m(\alpha\theta)^{n/\alpha}}{e^{-\theta S(u)}} \big\}} \bigg]    \\
&=~ -\sum_{|u|=n} \frac{e^{-\theta S(u)}}{m(\alpha\theta)^{n/\alpha}}
\int_{\R \setminus [-e^{\theta S(u)} m(\alpha\theta)^{n/\alpha} \tau,e^{\theta S(u)} m(\alpha\theta)^{n/\alpha} \tau]} y\, \dProb(V\leq y).
\end{align*}
Using \eqref{eq:tails of V} and integration by parts, we infer
\begin{equation*}   \textstyle
\int_{\R \setminus [-x,x]} y \, \dProb(V\leq y)
\sim \frac{\alpha (c_1-c_2)}{\alpha-1} x^{1-\alpha} \quad   \text{as } x \to \infty.
\end{equation*}
This asymptotic relation together with \eqref{eq:Biggins:1998 for Y_u^beta} implies that, as $n \to \infty$,
\begin{align*}
\sum_{|u|=n} & \E_n\bigg[\frac{e^{-\theta S(u)}V^{(u)}}{m(\alpha\theta)^{n/\alpha}}
\1_{\big\{\big|\frac{e^{-\theta S(u)}V^{(u)}}{m(\alpha\theta)^{n/\alpha}}\big| \leq \tau\big\}} \bigg]    \\
&=~ -\sum_{|u|=n} \frac{e^{-\theta S(u)}}{m(\alpha\theta)^{n/\alpha}}
\int_{\R \setminus [-e^{\theta S(u)} m(\alpha\theta)^{n/\alpha} \tau,e^{\theta S(u)} m(\alpha\theta)^{n/\alpha} \tau]}y\, \dProb(V\leq y)   \\
&\sim~ -\sum_{|u|=n} \frac{e^{-\theta S(u)}}{m(\alpha\theta)^{n/\alpha}}
\frac{\alpha (c_1-c_2)}{\alpha-1} \big(e^{\theta S(u)} m(\alpha\theta)^{n/\alpha} \tau \big)^{1-\alpha}    \\
&=~ -\tau^{1-\alpha} \frac{\alpha (c_1-c_2)}{\alpha-1}
\sum_{|u|=n} \frac{e^{-\alpha\theta S(u)}}{m(\alpha\theta)^{n}}
~\to~ -\tau^{1-\alpha} \frac{\alpha (c_1-c_2)}{\alpha-1} W(\alpha\theta) \quad \text{a.\,s.}
\end{align*}
This proves \eqref{eq:a}.
\end{proof}

\begin{proof}[Proof of Theorem \ref{Thm:stable fluctuations of Biggins's martingale}]
We show that, for any $r \in \N_0$,
\begin{align*}
\kappa^{-n/\alpha} \big((W\!-\!W_n),\ldots,
\kappa^{r/\alpha}(W\!-\!W_{n-r}) \big) \distto
W(\alpha\theta)^{1/\alpha}\left(U_0,\ldots,U_r\right).
\end{align*}
By the Cram\'er-Wold device, this is equivalent to proving the
following: for any $\beta_0,\ldots,\beta_r$ and $t \in \R$,
\begin{align}
&\lim_{n\to\infty}\E\bigg[\exp\!\bigg(\imag t\sum_{j=0}^r\beta_j\kappa^{-(n-j)/\alpha}(W\!-\!W_{n-j})\bigg)\bigg]   \notag  \\
&~=\E\bigg[\!\exp\!\bigg(\imag
tW(\alpha\theta)^{1/\alpha}\sum_{j=0}^r\beta_j U_{j}\bigg)\bigg]
=\E\bigg[\Phi(\gamma_0 W(\alpha\theta)^{1/\alpha}t)\prod_{i=1}^r
\Psi\big(\gamma_i W(\alpha\theta)^{\frac1\alpha}t\big)\bigg],
\label{eq:convergence ch.f. fdd}
\end{align}
where
\begin{equation}    \label{eq:gamma}
\gamma_i \defeq \sum_{j=i}^r\beta_j\kappa^{(j-i)/\alpha}
\end{equation}
for $i=0,\ldots,r$ and, for $t \in \R$,
\begin{align*}
\Phi(t)
&\defeq \E\big[\exp(\imag tU_0)\big]
=\exp\!\Big(\frac{\Gamma(2\!-\!\alpha)}{\alpha-1}\frac{c|t|^\alpha}{1-\kappa}
\Big(\!\cos\!\Big(\frac{\pi\alpha}{2}\Big)- \imag \sin\!\Big(\frac{\pi\alpha}{2}\Big)\sign(t)\Big) \Big),   \\
\Psi(t)
&\defeq \E\big[\exp(\imag tQ_1)\big]
=\exp\!\Big(\frac{\Gamma(2\!-\!\alpha)}{\alpha-1} c|t|^\alpha\Big(\!\cos\!\Big(\frac{\pi\alpha}{2}\Big)- \imag \sin\!\Big(\frac{\pi\alpha}{2}\Big)\sign(t) \Big) \Big)
\end{align*}
(see \eqref{eq:chf of Q_1} and \eqref{eq:chf U_0}).
Using the representation
\begin{equation*}
U_j=\kappa^{j/\alpha}U_0+\kappa^{(j-1)/\alpha}Q_1+\ldots+\kappa^{1/\alpha}Q_{j-1}+Q_j
\end{equation*}
for $j \in \N$, we obtain
\begin{equation*}
\sum_{j=0}^r \beta_j
U_{j}=\sum_{j=0}^r\beta_j\kappa^{j/\alpha}U_0+\sum_{i=1}^r\sum_{j=i}^r\beta_j\kappa^{(j-i)/\alpha}
Q_i
\end{equation*}
which justifies the last equality in \eqref{eq:convergence ch.f. fdd}.

Let $n \geq r$. For notational convenience, we set $\beta_j = 0$
for $j<0$. Then we have
\begin{align*}
\sum_{j=0}^r&\, \beta_j\kappa^{-(n-j)/\alpha}(W-W_{n-j})=\sum_{j=0}^r \beta_j\kappa^{-(n-j)/\alpha}\sum_{i\geq n-j}(W_{i+1}-W_{i})\\
&=\sum_{i\geq n-r}(W_{i+1}-W_i)\sum_{j=n-i}^r\beta_j\kappa^{-(n-j)/\alpha}\\
&=\sum_{i\geq 0}(W_{i+1+n-r}-W_{i+n-r})\sum_{j=r-i}^r \beta_j\kappa^{-(n-j)/\alpha}\\
&={m(\alpha\theta)}^{-(n-r)/\alpha}\sum_{|u|=n-m}e^{-\theta
S(u)}\sum_{i\geq 0}(W_{i+1}^{(u)}-W_{i}^{(u)})\sum_{j=r-i}^r
\beta_j\kappa^{-(r-j)/\alpha}.
\end{align*}
The last expression equals $\Theta_{n-r}$ defined in \eqref{xn}
with
\begin{equation*}
V=\sum_{i\geq 0} \kappa^{-i/\alpha}\gamma_{r-i}(W_{i+1}-W_{i}),
\end{equation*}
where for the negative integers we set $\gamma_{-i}\defeq\kappa^{i/\alpha}\gamma_0$.
Observe that the so defined $V$ is centered. Further, since the sequence $a_i \defeq
\kappa^{-i/\alpha}\gamma_{r-i}$ is eventually constant, by Theorem \ref{Thm:tail of W(theta)}, the distribution of $V$ satisfies
\eqref{eq:tails of V} with
\begin{align*}
c_1=c \sum_{i=-\infty}^r(\gamma_i^+)^{\alpha} \quad \text{and}
\quad c_2=c \sum_{i=-\infty}^r (\gamma_{i}^-)^{\alpha}.
\end{align*}
According to relation \eqref{eq:convergence to stable} from Lemma \ref{Lem:Theta_n convergence}
with these $c_1$ and $c_2$, we have
\begin{align*}
&\lim_{n\to\infty}
\E\bigg[\exp\bigg(\imag t\sum_{j=0}^r \beta_j\kappa^{-(n-j)/\alpha}(W\!-\!W_{n-j})\bigg)\bigg]   \\
&~=\E\bigg[\exp\bigg(B|t|^\alpha \sum_{j=-\infty}^r\Big(|\gamma_j|^{\alpha}\big(\cos (\tfrac{\pi\alpha}{2})-
\imag \sin (\tfrac{\pi\alpha}{2})\sign(\gamma_{j}t)\big)\Big) \bigg)\bigg]\\
&~=\E\bigg[\exp\!\Big((A|\gamma_0t|^\alpha\big(\!\cos(\tfrac{\pi\alpha}{2})\!-\!\imag \sin(\tfrac{\pi\alpha}{2})\sign(\gamma_0t)\big)\Big)\\
&\hphantom{~= \E\bigg[}~\cdot\prod_{j=1}^r \exp\bigg(B|\gamma_j t|^\alpha \big(\cos (\tfrac{\pi\alpha}{2})-\imag \sin (\tfrac{\pi\alpha}{2})\sign(\gamma_{j}t)\big) \bigg)\bigg]\\
&~= \E\bigg[\Phi(\gamma_0W(\alpha\theta)^{1/\alpha}t) \prod_{j=1}^r \Psi(\gamma_j
W(\alpha\theta)^{1/\alpha}t)\bigg],
\end{align*}
where
\begin{equation*}
A\defeq \frac{\Gamma(2-\alpha)}{\alpha-1}\frac{cW(\alpha\theta)}{1-\kappa}\quad\text{and}\quad B\defeq \frac{\Gamma(2-\alpha)}{\alpha-1}cW(\alpha\theta).\end{equation*}
The proof of Theorem \ref{Thm:stable fluctuations of Biggins's martingale} is complete.
\end{proof}

\subsection*{Acknowledgements}

The research of K.\,K.~and M.\,M.~was supported by DFG Grant ME
3625/3-1. K.\,K.~was partially supported by the National Science
Center, Poland (Sonata Bis, grant number DEC-2014/14/E/ST1/00588).
A part of the work was carried out during a visit of A.\,I.~to
Innsbruck University. The visit was supported by DFG Grant ME
3625/3-1 which is gratefully acknowledged.

\end{document}